\documentclass{amsart}
\usepackage{amsmath, amsfonts,amsthm,amssymb,amscd, 
verbatim,graphicx,color,multirow,booktabs, caption,tikz,tikz-cd, mathdots,bm}

\usepackage{tikz-cd}

\usepackage[breaklinks, colorlinks, citecolor=blue]{hyperref}

 \usepackage{graphicx}

\usepackage{latexsym}
\usepackage{longtable}
\usepackage{epsfig}
\usepackage{amsmath}
\usepackage{hhline}
\usepackage{comment}

\DeclareMathOperator{\soc}{soc}
\DeclareMathOperator{\aut}{Aut}

\DeclareMathOperator{\GL}{GL} 
\DeclareMathOperator{\GaL}{\Gamma L} 
\DeclareMathOperator{\agaL}{A\Gamma L}

\DeclareMathOperator{\PSL}{PSL}

\DeclareMathOperator{\GF}{GF}
\DeclareMathOperator{\gal}{Gal}

\DeclareMathOperator{\sym}{Sym}

\newcommand{\BB}{\mathcal{B}}

\newtheorem{thm}{Theorem}[section]
\newtheorem{cor}[thm]{Corollary}
\newtheorem{lemma}[thm]{Lemma}
\newtheorem{prop}[thm]{Proposition}

\theoremstyle{definition} 

\theoremstyle{definition}

\numberwithin{equation}{section}

\renewcommand{\footnote}{\endnote}
\newcommand{\ignore}[1]{}\makeglossary

\begin{document}
	\title[]{IBIS soluble linear groups} 

\maketitle

\begin {center}
\bf{Andrea Lucchini$^1$ and Dmitry Malinin$^2$} \footnotetext [1] { Dipartimento di Matematica 
\lq\lq Tullio Levi-Civita", Universit\`a degli Studi di Padova, Via Trieste 63, 35121, Padova, Italy; email address: lucchini@math.unipd.it }
\footnotetext [2] { Dipartimento di Matematica 
\lq\lq Tullio Levi-Civita", Universit\`a degli Studi di Padova, Via Trieste 63, 35121, Padova, Italy; email address: dmalinin@gmail.com }
\end {center}
\bigskip


\begin{abstract}
	Let $G$ be a finite permutation group on $\Omega.$
	An ordered sequence $(\omega_1,\dots, \omega_t)$ of elements of $\Omega$
	is an irredundant
	base for $G$ if the pointwise stabilizer is trivial and no point is fixed by the stabilizer of its predecessors. If all irredundant
	bases of $G$ have the same cardinality, $G$ is said to be an IBIS group.
In this paper we give a classification  
of quasi-primitive soluble irreducible IBIS linear groups, and we also describe nilpotent and metacyclic IBIS linear groups and IBIS linear groups of odd order. 
\end{abstract}

\section{Introduction}
 Let $G\le \sym(\Omega)$ be a finite permutation group.  A subset $\BB$ of $\Omega$ is a \emph{base} for
$G$ if the pointwise stabilizer $G_{(\BB)}$ is trivial and we denote by $b(G)$ the
minimal size of a base for $G$.   An ordered sequence of elements of $\Omega$, $\Sigma:=(\omega_1,\dots, \omega_t)$, is \emph{irredundant} for $G$ if no point in $\Sigma$ is fixed by the stabiliser of its predecessors. Moreover $\Sigma$ is an \emph{irredundant base} of $G$, if it is a base and it is irredundant. In particular, such an irredundant base provides the following stationary chain
$$G>G_{\omega_1}>G_{(\omega_1, \omega_2)}>\dots>G_{(\omega_1,\dots, \omega_{t-1})}>G_{(\omega_1,\dots, \omega_t)}=1,$$
where the inclusion of subgroups are strict.


Irredundant bases for a permutation group possess some of the features of bases in a vector space.
Indeed, $\BB$ is a basis for a vector space $V$ if and only if it is an irredundant base for the general linear group $\GL(V)$ in its natural action on $V.$ 
However, some familiar properties of bases in vector spaces do not extend to bases for permutation groups: irredundant bases for groups in general are not preserved by reordering, and they can have different sizes.

In~\cite{ibis}, Cameron and Fon-Der-Flaass showed that all irredundant bases for a permutation group $G$ have the same size if and only if all the irredundant bases for $G$ are preserved by re-ordering. Groups satisfying one of the previous equivalent properties are called \emph{Irredundant Bases of Invariant Size groups}, \emph{IBIS groups} for short.
Moreover, Cameron and Fon-Der-Flaass~\cite{ibis} also proved that for a permutation group $G$ to be IBIS is a necessary and sufficient for the irredundant bases of $G$ to be the base of a combinatorial structure known with the name of matroid.  If this condition hold, then $G$ acts geometrically on the matroid and when $G$ acts primitively and is not cyclic of prime order, then the matroid is geometric (see~\cite{ibis} for more details). This brought Cameron to ask for a possible classification of the IBIS groups. As explained by Cameron himself in~\cite[Section 4-14]{cam}, there is no hope for a complete classification of IBIS groups, when the cardinalities of the bases are large. But it might be reasonable to pose this question for primitive groups. The first attempt of classifying finite primitive IBIS groups has been done in \cite{lmm}. Their
approach is via the O'Nan-Scott classification of primitive groups.
\begin{thm} \cite[Theorem 1.1]{lmm} Let $G$ be a primitive IBIS group. Then one of the following holds:
	\begin{enumerate}
\item $G$ is of affine type,
\item $G$ is almost simple,
\item $G$ is of diagonal type.
\end{enumerate}
Moreover, G is a primitive IBIS group of diagonal type if and only if it belongs to the infinite family of diagonal groups
$\{\PSL(2,2^f
	) \times \PSL(2,2^f
	) | f\in \mathbb N, f\geq 2\}$ having degree $|\PSL(2,2^f
)| = 2^f
(4^f-1).$
\end{thm}

In the light of this result, the problem of understanding finite primitive IBIS groups is reduced to affine groups and to almost simple groups. There is some hope for dealing with the second family. Indeed, suppose that $G$ is an almost simple primitive IBIS group with $b(G) = 2.$ In particular, all bases of $G$ have cardinality 2. Thus, for any two distinct points $\alpha$ and $\beta$ in the domain of $G$ we have $G_\alpha \cap G_\beta=1.$ Therefore $G$ is a Frobenius group, contradicting the fact that $G$ is almost simple. Thus, if $G$ is an almost simple primitive IBIS group, then $b(G) \geq 3$ and one can get useful information from the work of Burness, O'Brien, Liebeck, Saxl, Shalev and Wilson on the base size of primitive permutation groups 
(see for instance \cite{b1}, \cite{b2}, \cite{b3}, \cite{b4}). A first result in this direction has been obtained in \cite{lp}, where the authors determine the almost simple primitive IBIS groups having socle an alternating group.

Dealing with the primitive groups of affine type seems very difficult. In this case $G$ is a semidirect product $G= V \rtimes X,$ where $X$ is an irreducible subgroup  of $\GL(V)$. It follows immediately from the definition that $G$ is an IBIS permutation group if and only if $X$ induces an IBIS permutation group  on the set $V\setminus \{0\}$, in which case we will say that $X$ is a linear IBIS group, or more precisely, that $X$ is an IBIS subgroup of $\GL(V).$

In this paper we are interested in collecting information on the IBIS irreducible linear groups that are soluble. Given $G\leq \GL(V)$, we will denote by $b(G)$ the smallest cardinality of a base for the action of $G$ 
on $V\setminus \{0\}.$ By the main result in \cite{as}, $b(G)\leq 3$ for any soluble irreducible linear group $G.$ 
Notice that $b(G)=1$ if and only if $G$ has a regular orbit on $V$ and 
that in this case  $G$ is IBIS if and only if it  is a Frobenius complement in $V\rtimes G$, i.e.  if and only if $G$ acts fixed-point-freely on $V.$ In a very resent paper Yang, Vasil'ev and Vdovin \cite{yvv}
proved that an irreducible quasi-primitive soluble subgroup of $\GL(V)$ which is not metacylic has a regular orbit, except for a few \lq\lq small\rq\rq\ cases. These cases have been completely classified in a subsequent paper by Holt and Yang \cite{HY}, and so it is possible to produce with MAGMA the  list of all the primitive irreducible non-metacyclic groups $G$ that are IBIS with $b(G)\neq 1.$ More precisely we can state the following results:

\begin{thm}\label{qpri}
	If $G$ is a soluble irreducible quasi-primitive IBIS subgroup of $\GL(V),$ then one of the following occurs:
	\begin{enumerate}
		\item $G$ is metacyclic;
		\item $G$ is a Frobenius complement;
		\item $G$ is one of the 39 exceptional linear groups listed in Table  \ref{table1}.
	\end{enumerate}
\end{thm}

The metacyclic case is investigated in Section \ref{metac}. If $V$ is a finite vector space of dimension $n$ over $\GF(q),$ where $q$ is a prime power, we denote by $\Gamma(q^n)=\Gamma(V)$ the semilinear group of $V,$
 i.e. $$\Gamma(q^n)=\{x \mapsto ax^\phi\mid x\in \GF(q^n), a \in (\GF(q^n))^*, \phi \in \gal(\GF(q^n)/\GF(q))\}.$$ 
This group has a normal cyclic subgroup $N$ of order $q^n-1$ (a so called
Singer cycle) consisting of those elements with $\phi=1.$ The subgroup $N$ acts regularly on $V\setminus\{0\},$ and $\Gamma(q^n)/N$ is cyclic of order $n.$ Our main result is the following:


\begin{thm}\label{meta}
	Suppose that $G\leq \Gamma(q^n),$ with $n>1$ and $|G/G\cap N|=n.$ If $G$ is IBIS, then either $G$ is a Frobenius complement or there exists a prime divisor $s$ of $n$ such that, for every $v\in  
	V\setminus\{0\},$ the stabilizer $G_v$ has order $s.$ More precisely, if
	$\phi$ is the automorphism of $\GL(q^n)$ sending to $f$ to $f^q$ and $\phi b\in G$ for some $b \in N,$ then $G$ is IBIS if and only if there exists a divisor $r$ of $n$ such that:
	\begin{enumerate}
		\item $N^{q^r-1}\leq G;$
		\item for every proper divisor $t$ of $r,$ $b^{(q^t-1)/(q-1)}\notin N^{q^t-1}(N\cap G).$
				\item one of the following occurs:
				\begin{enumerate}
		\item $s=n/r$ is a prime and $b^{(q^r-1)/(q-1)}\in N\cap G;$
		\item $n=r$. 
		\end{enumerate}
	\end{enumerate}
	In case (a) the  stabilizers in $G$ have order $s,$ in case (b) $G$ is a Frobenius complement.
\end{thm}

In Section \ref{metac}, we will apply the previous theorem to produce some non-trivial examples of IBIS metacyclic linear groups. Two interesting corollaries of Theorem \ref{meta} are the following:
\begin{cor}
	Suppose $n>1.$ Then $\Gamma(q^n)$ is IBIS if and only if $n$ is a prime. 
\end{cor}

\begin{cor}\label{corodd}	Suppose that $G\leq \Gamma(q^n),$ with $n$ a prime and $|G/G\cap N|=n.$ Then $G$ is IBIS  if and only if one of the following occurs:
	\begin{enumerate}
		\item 
		$G=(N\cap G) \langle \phi \rangle,$ with  $N^{q-1}\leq G;$
		\item $G=(N\cap G) \langle \phi b\rangle,$ for some $b \in N$ such that $b\notin N^{q-1}(N\cap G).$
		\end{enumerate}
\end{cor}

For example consider the subgroup $G$ of $\Gamma(5^2)$ generated by $\phi$ and $a\in N$  with $|a|=3:$ since $|N^4|=6,$ $G$ is not IBIS.

The study of the IBIS property in the case of imprimitive soluble linear groups is much more difficult. Apparently there is no strategy for a reduction from the imprimitive to primitive case. However we are able to solve the problem in the particular case of nilpotent irreducible linear groups. First we describe two families of irreducible linear groups that are IBIS. The following is a consequence of Corollary \ref{corodd}.

\begin{prop}\label{uno}Let $r=2^b-1$ be a Mersenne prime.  Consider the following two subgroups of $\Gamma(r^2)=N\langle \phi\rangle:$
$P_1=O_2(N)\langle \phi \rangle$ (a semidihedral group of size $2^{b+2}$) 
	and $P_2=(O_2(N))^2\langle \phi \rangle$ (a dihedral group of size $2^{b+1}$). 
	If $G=PC\leq \Gamma(r^2)$ with $P\in \{P_1, P_2\}$ and $C\leq O_{2^\prime}(N),$ then $G$ is an irreducible nilpotent IBIS subgroup of $\Gamma(r^2)$ and $b(G)=2.$
\end{prop}

\begin{prop}\label{due}
	Let $W$ be the additive group of a finite field $F$ and let $H$ be the multiplicative group of $F$, acting on $W$ by scalar multiplication and consider the wreath product $X=H\wr S_2$ in its imprimitive action on $V=W^2.$ Let $G$ be a subgroup of $X$ satisfying the following conditions:
	\begin{enumerate}
		\item $\langle (1,2), (h,h^{-1})\mid h\in H\rangle \subseteq G$ for every $h\in H$;
		\item $(k,1)\in G$ for some $1\neq k \in H.$
	\end{enumerate}
	Then $G$ is an irriducible IBIS subgroup of $\GL(V),$ with $b(G)=2.$
	In particular if $H$ is a 2-group, then $G$ is a nilpotent irreducible IBIS linear group.
\end{prop}
Now we may state our result on nilpotent irreducible IBIS linear groups. 

\begin{thm}\label{ibisnilp}
	If $G$ is a nilpotent IBIS irreducible subgroup of $\GL(V),$ then one of the following occurs:
	\begin{enumerate}
		\item $G$ acts fixed-point-freely on $V.$
		\item $G$ is as in Proposition \ref{uno}.
		\item $G$ is as in Proposition \ref{due}.
	\end{enumerate}
\end{thm}

In Section \ref{odd} we consider the case of IBIS linear groups of odd order. \begin{thm}\label{ibisodd}
	Let $G$ be an IBIS irreducible subgroup of $\GL(V)$. If $|G|$ is odd, then either $G$ acts fixed-point-freely on $V$ or $G$ is a primitive metacyclic group. 
\end{thm}

\begin{cor}\label{ibisoddpr}
Assume that $G$ is a soluble primitive permutation group of degree $n$ and odd order. If $G$ is an IBIS permutation group, then either $G$ is a Frobenius group or $G\leq \agaL(1,n).$
\end{cor}

\medskip
\noindent {\bf Acknowledgements.} {\sl The authors are grateful to Pablo Spiga for the computations performed for finding exceptional IBIS groups using MAGMA. 

\noindent The second author is grateful for the research grant from the University of Padova.}

\section{The 39 exceptions in Theorem \ref{qpri}}

In Table~\ref{table1} we collect the quasi-primitive solvable groups that do not have a regular orbit and that are IBIS. Here we give some information in how to read Table~\ref{table1}. Essentially, we follow the notation in Table~4.1 in the work of Holt and Yang \cite{HY}. We have divided Table~\ref{table1} in 9 columns. The first column is the same ``No. column" as in~\cite{HY} and simply denotes the line numbering (in our case restricted to the cases where some IBIS group does exist) for the files in the supplementary matherial in~\cite{HY}. The ${\bf e}$, ${\bf p}$, ${\bf d}$, ${\bf a}$, ${\bf b}$ have the same meaning as the respective columns in~\cite[Table~4.1]{HY} and this meaning is explained in~\cite[page 141]{HY}. The {\bf Nr.} indicates the number of groups that are IBIS in the corresponding line and the {\bf Grp. ordrs.} indicates the respective orders. In the {\bf Cmmts.} column we are giving the  position in the corresponding file of the IBIS groups: the first two rows also indicate whether the groups are of type $E^+$ or $E^-$, as indicated in~\cite{HY}.

\begin{table}[!ht]
	\begin{tabular}{c|c|c|c|c|c|c|c|c|c}
		{\bf No.}&${\bf e}$&${\bf p}$&${\bf d}$&${\bf a}$&${\bf b}$&{\bf Nr. }&{\bf Grp. ordrs.} &{\bf Cmmnts}&{\bf Bs. Sz.}\\
		19&4&3&4&1&1&4&$640$, $192$, $320$, $160$& 2,3,5, 6 $E^-$&2, 2, 2, 2\\
		19&4&3&4&1&1&4&$1152$, $1152$, $576$, $192$& 2,4,5, 14 $E^+$&3, 3, 3, 2\\
		48&3&2&6&2&1&1&$648$& 3&3\\
		62&2&3&2&1&1&2&$48$, $24$& 1, 2&2, 2\\
		63&2&5&2&1&1&2&$96$, $48$& 1, 2&2, 2\\
		64&2&7&2&1&1&2&$144$, $72$& 1, 2&2,2\\
		65&2&3&4&2&1&7&$48$, $96$, $96$, $96$&&2, 2, 2, 2\\
		&&&&&&&$96$, $192$, $192$& 3,4,5,7,8,10, 11&2, 2, 2\\
		66&2&11&2&1&1&2&$240$, $48$& 1, 2&2, 2\\
		67&2&13&2&1&1&2&$288$, $144$& 1, 2&2, 2\\
		68&2&17&2&1&1&3&$384$, $96$, $192$& 1, 2, 3&2, 2, 2\\
		69&2&19&2&1&1&2&$432$, $144$& 1, 2&2, 2\\
		71&2&5&4&2&1&6&$576$, $144$, $288$&&2, 2, 2\\
		&&&&&&&$288$, $288$, $576$& 2,5,8,11,12, 13&2, 2, 2\\
		72&2&3&6&3&1&2&$1872$, $936$& 1, 2&2, 2\\
	\end{tabular}
	\caption{Parameters of IBIS quasi-primitive solvable groups that do not have a regular orbit.}\label{table1}
\end{table}

The computational work to produce the table has been organized as follows:
\begin{enumerate}	
\item Exhaustive search of bases having size 2. All cases where small enough, so this step  was easy and only a matter of time.
\item For the groups having a base of size 2 (the most frequent sitation) we started a random search for bases having size 3. In all cases where this was successful, the group was not IBIS and hence it could be omitted from further analysis.
\item For the groups passing the previous step, if these groups are small, we checked extensively for bases having 3. If the groups are large, for each orbit
 we constructed the action of $G$ on this orbit and checked extensively for bases having size 3 in this action.
 \item Similarly, for the few cases where there is no base of size 2, we check the existence of bases of size 4.
\end{enumerate}

\section{Metacyclic groups}\label{metac}

\begin{lemma}\label{ibprim}
	Suppose that $G\leq \Gamma(q^n),$ with $n>1$ and $|G/G\cap N|=n.$ If $G$ is IBIS, then either $G$ is a Frobenius complement or, for every $v\in  
V\setminus\{0\},$ the stabilizer $G_v$ has order prime order.
\end{lemma}

\begin{proof}
Let $\Gamma=\Gamma(q^n).$ We identify the elements of $N$ and the elements of $W:=V\setminus \{0\}$ with the non-zero elements of the field $\GF(q^n).$ 
Notice that, for every $a\in W,$ we have $\Gamma_a=\langle \phi a^{1-q}\rangle.$
Moreover, if $m$ divides $n,$ then $(\phi a^{1-q})^m=\phi^m a^{1-q^m}.$
Suppose that $G$ is  IBIS, but not a Frobenius complement. 
Let $a\in W.$ We want to prove that the stabilizer $G_a$ of $a$ in $G$ has prime order. Since  $G_a$  coincides with the stabilizer of 1 in the subgroup $G^{a^{-1}}$ of $\Gamma$ obtained by conjugation with $a^{-1},$ we may assume
 $a=1.$ Hence there exists a proper divisor $r$ of $n$ with $G_1=\langle \phi^r \rangle$.  Let $s=|G_1|=n/r.$ Notice that if $N=\langle c \rangle,$ then
$\Gamma_c \cap \langle \phi \rangle=1.$ So in particular $G_c \cap G_1=1,$ and $(1,c)$ is an irredundant base. Suppose that $s$ is not a prime, and write $s=tu,$ with $t,u >1.$ Then $n=rs=rtu$, so $q^{rt}-1$ divides $q^n-1$ and if $d \in N$ has order $q^{rt}-1,$ then
$\phi^{rt}$ is a non-trivial element of $G_d$, while $\phi^{rt}\notin G_c.$ But then also $(1,d,c)$ is an irredundant base, in contradiction with the fact that $G$ is IBIS.	
\end{proof}

\begin{lemma}\label{ibse}
	Suppose that $G\leq \Gamma(q^n),$ with $n>1$ and $|G/G\cap N|=n.$ If $G$ is IBIS, then either $G$ is a Frobenius complement or there exists a prime $s$ such that $|G_v|=s$ for every $v\in  
	V\setminus\{0\}.$ 
\end{lemma}

\begin{proof}
Suppose that $G$ is  IBIS, but not a Frobenius complement. By the previous lemma, $|G_1|=s_1$ for some prime $s_1.$ Let now $a \in V \setminus \{0\}$, and assume $|G_a|=s_2.$ Let $r_1=n/s_1$ and $r_2=n/s_2$. Notice that $g_1=\phi^{r_1}$ and $g_2=\phi^{r_2}a^{1-q^{r_2}}$ belong to $G.$ Let $d=(r_1,r_2)$. There exist $x,y \in \mathbb Z$ with $d=xr_1+yr_2.$ Hence, setting $t=1+q^{r_2}+\dots+q^{r_2(y-1)}$ and $u=t(1-q^{r_2})/(1-q^d),$ we have
$$g_1^xg_2^y=\phi^d a^{(1-q^{r_2})t}=\phi^d (a^u)^{1-q^d}\in G_{a^u}.$$
Hence $n/d$ divides $|G_{a^u}|$ and therefore, by Lemma \ref{ibprim}, $n/d$ is a prime. This is possible only if $d=r_1=r_2$, i.e. if $s_1=s_2.$
\end{proof}

\begin{proof}[Proof of Theorem \ref{meta}]
	The first part of the statement follows from Lemma \ref{ibse}.
	
	Assume that $G$ is an IBIS subgroup of $\Gamma=\Gamma(q^n)$ and that $\phi b \in G$ for some $b\in N.$ We identify the elements of $N$ and the elements of $W:=V\setminus \{0\}$ with the non-zero elements of the field $\GF(q^n)$ and we set $H=G\cap N.$ 
	Either $G$ is a  Frobenius complement or there exists a prime $s$ such that all the stabilizers in $G$ have order $s.$ We set $r=n$ in the first case, $r=n/s$ in the second. If $G$ is a Frobenius complement, then $N^{q^r-1}=N^{q^n-1}\in G.$ Otherwise $G_1=\langle \phi^r \rangle$ and therefore
	$\phi^r\in G=H\langle \phi b \rangle,$ which implies $b^{(q^r-1)/(q-1)}\in H.$
 Moreover, for every $a\in W,$ $G_a=\langle \phi^r a^{1-q^r}\rangle \leq G$. This implies in particular that $N^{q^r-1}\leq H.$
 Finally notice that in any case, if $t$ is a proper divisor of $r$ then, for every $a\in V\setminus \{0\},$ $\phi^t a^{1-q^t}\notin \Gamma_a \setminus G$: this means 
$$\phi^t a^{1-q^t}\neq (\phi b)^t= \phi^t b^{(q^t-1)/(q-1)} \mod H,$$ for every $a\in N$,
or equivalently $b^{(q^t-1)/(q-1)}\notin  N^{q^t-1}(G\cap N).$ 

Conversely, if $G=H\langle \phi b\rangle$ satisfies conditions (1), (2) and (3), then $|G_a|=n/r$ for every $a \in V\setminus \{0\}$, and therefore $G$ is IBIS.
\end{proof}


Consider for example the following subgroup of $\Gamma(q^4)$:
$G=\langle \phi a, a^{q+1} \rangle$, with $|a|=q^4-1.$ 
We have $N^{q^2-1}\leq G.$ Moreover, if $q$ is odd, then $a\notin N^{q-1}(G\cap N)=N^{q-1}N^{q+1}=\langle a^2\rangle.$ But then it follows from Theorem \ref{meta} that if $q$ is odd, then $G$ is an IBIS linear group of order $4(q^3-q^2+q-1),$ in which the stabilizer of every nonzero vector has order 2.





\section{Nilpotent irreducible IBIS linear groups} 

A result useful to investigate the IBIS irreducible linear groups $G$ with $b(G)=2$ is the following:

\begin{prop}\label{trans}
	Let $G$ be an IBIS irreducible subgroup of $\GL(V)$. If $b(G)=2$ and $V \cong W^n$ is an imprimitive decomposition for the action of $G$, then the stabilizer $H$ of $W$ in $G$ acts transitively on $W\setminus \{0\}.$
\end{prop}


\begin{proof}
	We may identify	$G$ with a subgroup of the wreath product $H \wr S_n.$ 
	Assume that $H$ is not transitive on $W \setminus \{0\}.$ Then there exists an orbit $\Omega$ for the action of $H$ on $W \setminus \{0\},$ with $|\Omega| \leq (|W|-1)/2.$ Take $u \in \Omega$ and	let $\alpha=(u,0,\dots,0)\in V=W^n.$ For any
	$2\leq i\leq n$ and $w \notin \Omega \cup\{0\},$ let $$\beta_{i,w}=(0,\dots,0, w,0,\dots,0) {\text{ and }}\gamma_{i,w}=(u,0,\dots,0, w,0,\dots,0)$$ (where $w$ is the entry in the $i$-position). If $y \in G_{\gamma_{i,w}},$ then $y=(h_1,\dots,h_n)\sigma,$ with $uh_1=u, wh_i=w$ and $1\sigma=1,$ $i \sigma=i.$ In particular $G_{\gamma_{i,w}}=G_\alpha \cap G_{\beta_{i,w}}.$
	Since $b(G)=2,$ $G_{v}\neq 1$ for every $v\in V,$ and since $G$ is IBIS, if
	$G_{v_3}=G_{v_2}\cap G_{v_1},$ then $G_{v_1}=G_{v_2}=G_{v_3}.$
	So
	$G_\alpha = G_{\beta_{i,w}},$ for any $w \notin \Omega \cup\{0\},$ and any $i\geq 2.$ In particular, if $g\in G_\alpha,$ then $g=(x,y_2,\dots,y_n)$ with $x\in H_u$ and $y_j\in G_w$ for every $w \notin \Omega.$
	But then the number of vectors in $W$ fixed by $y_j$ is at least $(|W|+1)/2,$
	and therefore $y_j=1$ by \cite[Theorem 1]{gm}.
		This implies that there exists $1\neq K \leq H$ such that
	$G_\alpha=\{(k,1,\dots,1)\mid k \in K\}.$ On the other hand, $G$ contains an elements $z=(h_1,\dots,h_n)\tau$ with $1\tau=2$ and $K_\alpha^z=\{(1,k^{h_1},1,\dots,1)\mid k \in K\}\leq G_\alpha,$ a contradiction.
\end{proof}
\begin{thm}\label{alpiudue}\cite[Theorem 1.1]{HP}  $b(G)\le 2$ for every irreducible nilpotent subgroup of $\GL(V).$
	\end{thm}


\begin{proof}[Proof of Proposition \ref{due}]
	Let $A=\{(h,0)\mid h \in H\}\cap G, B=\{(0,h)\mid h \in H\}\cap G$ and let $v=(w_1,w_2)\neq (0,0)\in V.$ If $w_1=0,$ then $G_v=A,$ if $w_2=0,$ then $G_v=B$ and if neither $w_1=0$ nor $w_2=0,$ then $G_v=\langle (h,h^{-1})(1,2)\rangle$ where $h$ is the unique element of $H$ with $v_1h=v_2.$ In the last case $|G_v|=2$ and $G_v \cap A = G_v \cap B=1.$ 
	Hence, if $v_1$ and $v_2$ are non-zero elements of $V$, then either $G_{v_1}=G_{v_2}$ or $G_{v_1}\cap G_{v_2}=1.$
\end{proof}

\begin{proof}[Proof of Theorem \ref{ibisnilp}]
If $G$ has a regular orbit on $V$, then (1) holds. So assume that $G$ does not have a regular orbit on $V.$ By Theorem \ref{alpiudue}, $b(G)=2.$ 

If $G$ is primitive, then
by \cite[Theorem 2.4]{regularorbitnil} (see also \cite[Remark 2.6]{regularorbitnil}) $G$ is as in Proposition \ref{uno}.

So we may assume that $V$ has an imprimitive decomposition $V=W_1\times \dots \times W_n$, with $W\cong W_i$ for $1\leq i \leq n$, and  that $G$ permutes primitively the factors $W_1,\dots,W_n.$ Let $H$ be the nilpotent, irreducible subgroup of $\GL(W)$ induced by the action on $W_1$ of the setwise stabilizer of $W_1$ in $G$ and $T$ the nilpotent primitive permutation group induced by the permutation action of $G$ on the set $\{W_1,\dots,W_n\}$. It must be that $n$ is a prime and  $T=\langle\sigma\rangle$, with  $\sigma=(1,2,\dots,n)$, and
$G\leq H \wr \langle \sigma \rangle.$
Let $X:=G\cap H^n$. Notice that $X$ is a subdirect product of the base subgroup $H^n$, i.e., if $\pi_i:X\to H$ is the projection on the $i$-th component, then $X^{\pi_i}=H$ for each $i=1,\dots,n$. For any $w\in W,$ let
$v_w=(w,0,\dots,0) \in V.$ If $w\neq 0,$ then $G_{v_w}\leq X$ and $(G_{v_w})^\pi=H_w.$ In particular if $(w_1,\dots,w_t)$ is an irredundant base
for the action of $H$ on $W\setminus \{0\},$ then $(v_{w_1},\dots,v_{w_t})$ is an irredundant base for the action of $G$ on $V\setminus \{0\}.$ Hence the assumption that $G$ is an IBIS subgroup of $\GL(V)$ implies that $H$ is an IBIS subgroup of $\GL(W).$ By Theorem \ref{alpiudue}, $b(H)=2.$ So we distinguish two cases:

\noindent a) $b(H)=1.$ Fix  $0\neq w\in W$, let $\beta_i=(w,\dots,w,0,w,\dots,w)\in V$, with $0$ in the $i$-th position,
and $\gamma_i=(0,\dots,0,w,0,\dots,0)\in V$ with $w$ in the $i$-th position. As we are assuming that $G$ is IBIS with $b(G)=2$, $G_{\beta_i}\ne 1$.  As $H_{w}=1$, we have $G_{\beta_i}=\{(1,\dots,1,s,1,\dots,1)|s\in S_i\}$ for some nontrivial subgroup $S_i$ of $H$. Moreover, if $g\in G_{\gamma_i},$ then $g\in H^n$ and $g^{\pi_i}=1.$ If $n\ge 3$ then 
$1\ne G_{\beta_3}\le G_{\gamma_1}\cap G_{\gamma_2}<G_{\gamma_1}$, as $G_{\beta_2}\leq G_{\gamma_1}$ and $G_{\beta_2}\not\le G_{\gamma_2}$. This contradicts the fact that $G$ is IBIS and $b(G)=2$. Hence $n=2.$ By Proposition \ref{trans}, the action of $H$ on $W\setminus \{0\}$ is transitive. Moreover, since $H$ is IBIS and $b(H)=1,$ this action is also regular. So we may identify $W$ with the additive group of the field $F$ with $r^k$ elements, being $r$ a prime, and $H$ with the multiplicative group of this field. Let $f$ be a non-zero element of $F$
and $v=(1,f)\in W^2.$ Since $b(G)=2,$ we have
$1\neq (t_1,t_2)\sigma \in G_v$ for some $t_1, t_2 \in H,$ with $\sigma=(1,2).$ This implies $t_1=f,$ $t_2=f^{-1}$ and therefore $(f,f^{-1})\sigma \in G.$ But now let $u=(f,f)\in V.$ 
Again we have
$1\neq (u_1,u_2)\sigma^j \in G_u$ for some $u_1, u_2 \in H$ and $j \in \{0,1\}.$
The only possibility is $u_1=u_2=1$ and $j=1.$ So $\langle (1,2), (h,h^{-1})\mid h\in H\rangle \subseteq G$. Since $G$ is nilpotent,
$\langle (1,2), (h,h^{-1})\mid h\in H\rangle$ is also nilpotent, and this implies that $r^k-1$ is a 2-power.  
 Finally let $z=(0,1)\in V.$ Again $G_z\neq 1,$ and this is possible only if $(h,1)\in G$ for some $1\neq h \in H.$ We have so proved that $G$ is as described in Proposition \ref{due}.

\noindent b) $b(H)=2.$ 
 Let $0\neq w\in W$. There exists a nontrivial element $h\in H$ such that $w^h=w$. Since $X^{\pi_1}=H$, there exist $h_2,\dots,h_n\in H$ such that  $(h,h_2,\dots,h_n)\in X$. 
 Assume that $h_2=\dots=h_n=1$. Then, as $\langle\sigma\rangle$ acts transitively on $H_1,\dots,H_n,$ there exists a nontrivial element
 $t\in H$ such that $(1,t,1,\dots,1)\in X$. Take $u\in W$ such that $u^h \ne u$ and let $\beta=(w,0\dots,0), \gamma=(u,0,\dots,0)$. Then $1\ne G_{\beta}\cap G_{\gamma}< G_{\gamma} $, as $1\ne (1,t,1,\dots,1)\in  G_{\beta}\cap G_{\gamma}$ and $(h,1,\dots,1)\in G_{\beta}\setminus G_{\gamma}$. This is in contradiction with the fact that $G$ is IBIS. Thus there exists $i$ such that $h_i\ne 1$. Let $z\in W$ be such that $z^{h_i}\ne z$. Assume that $n\ne 2$. Let $\delta=(w,z_2,\dots,z_n)\in V$, where $z_i=z$ and $z_j=0$ if $j\ne i$. Then $G_\delta\le X$. As $G$ is an IBIS group and $b(G)\neq 1,$  we have that $G_\delta\ne 1$.  Moreover $1\ne G_\delta=X_\delta  <  X_{\beta}$, as $(h,h_2,\dots,h_n)\in  X_{\beta}\setminus X_\delta$, but this gives again a contradiction. We have so proved that $n=2$. Moreover, repeating the argument above, we deduce that $X$ contains no nontrivial elements of the type $(h,1)$, with $h\in H$. As $X$ is a subdirect product of $H^2$, there exists an automorphism $\alpha$ of $H$ such that $X=\{(h,h^\alpha)|\,h\in H\}$. Moreover
there exists $g=(1,k)\sigma \in G$ and $\alpha \in \aut(G)$ such that $G=\{(h,h^\alpha) \mid h \in H\}\langle g \rangle.$
 
  By Proposition \ref{trans}, $H$ is transitive on $W\setminus \{0\},$ and consequently $H$ is a primitive irreducible subgroup of $\GL(W).$ By \cite[Theorem 2.4]{regularorbitnil}  $H$ is as is Proposition \ref{uno}.
In particular
$W$ can be identified with the additive subgroup of the field $F$ with $r^2$ elements, where $r$ is a Mersenne prime and $H$ can be identified with a subgroup of $\Gamma=
\Gamma(r^2).$ Let $0\neq f \in F.$ We have $\Gamma_f \cong C_2$. On the other hand, since $b(H)=2,$ $H_f\neq 1,$ and therefore $H_f=\Gamma_f.$ By the transitivity of $H$ on $F\setminus \{0\},$ it follows $H=\Gamma.$
Suppose that there exists $h \in H$ such that $hh^\alpha k \neq 1$. There exists $0\neq v \in W$  with $hh^{\alpha} k\in H_v.$ Let $w=v^h$ and $y=(h,h^\alpha)(1,k)\sigma.$ Then $y \in G_{(v,w)}$ and $|y|>2.$
In particular $y^2=(hh^\alpha k, h^\alpha kh)\in  G_{(v,w)}\cap G_{(v,0)}.$
Since $G$ is IBIS, we must have $\{(t,t^\alpha) \mid t\in H_v\}=G_{(v,0)}=G_{(v,w)},$ in contradiction with the fact that $|G_{(v,0)}|=|H_v|=2$ and $|G_{(v,w}|\geq |y|>2.$
 But then
$h^\alpha=(kh)^{-1}$ for any $h \in H.$ This implies $k=1$
and $h^\alpha=h^{-1}$ for every $h\in H$. However this would implies that $\Gamma=H$ is abelian. This final contradiction implies that the case $b(H)=2$ cannot occur. \end{proof}

\section{IBIS primitive permutations groups of odd order}\label{odd}

\begin{proof} [ Proof of Theorem  \ref{ibisodd}]
By \cite[Theorem 1.3]{as}, $b(G)\leq 2.$ If $b(G)=1,$
	then $G$ acts fixed-point-freely on $V$. 
	So we may assume $b(G)=2.$ This implies that $G_v \neq 1$ for every $0\neq v\in V,$
and if $0\neq v_1, v_2,$ then either $G_{v_1}=G_{v_2}$ or $G_{v_1}\cap G_{v_2}=1.$	
Assume, by contradiction, that $G$ is not primitive. Then $G\leq H \wr S_n,$ with $H\leq \GL(W)$ and $V\cong W^n$ and $n\geq 2.$  Fix $0\neq u \in W$ and let $\alpha=(u,0,\dots,0).$ For any
	$2\leq i\leq n$ and $0\neq w \in W,$ let $\beta_{i,w}=(0,\dots,0, w,0,\dots,0)$ and $\gamma_{i,w}=(u,0,\dots,0, w,0,\dots,0)$ (where $w$ is the entry in the $i$-position). If $g \in G_{\gamma_{i,w}},$ then $g=(h_1,\dots,h_n)\sigma,$ with $\{1,i\}\sigma=\{1,i\}.$
	Since $|G|$ is odd, also $|\sigma|$ is odd. Thus $1\sigma=1,$ $i \sigma=i$ and consequently  $u^{h_1}=u, w^{h_i}=w$ and
	$1\neq G_{\gamma_{i,w}}=G_\alpha \cap G_{\beta_{i,w}}.$ Hence
	$G_\alpha = G_{\beta_{i,w}},$ for any $0\neq w$ and any $i\geq 2.$ But then there exists $1\neq X \leq H$ such that
	$G_\alpha=\{(x,1,\dots,1)\mid x \in X\}.$ On the other hand $G$ contains an elements $z=(y_1,\dots,y_n)\tau$ with $1\tau=2$ and $(G_\alpha)^z=\{(1,x^{y_1},1,\dots,1)\mid x \in X\}\leq G_\alpha,$ a contradiction. Finally, if $G$ is primitive, then it follows from Theorem \ref{qpri} that $G$ is metacyclic.
\end{proof}

To give an example of an IBIS linear group of odd order, choose an odd prime $r$ and consider the subgroup $G$ of $\Gamma(q^r)$ generated by $\phi$ and an element of $N$ of order $(q^r-1)/(q-1)$. Then $|G|=r(q^r-1)/(q-1)$ is odd and $G$ is IBIS by Corollary \ref{corodd}. Another example can be obtained  taking the subgroup $G:=\langle \phi a^2, a^{2c} \rangle$ of $\Gamma(7^9)$ with $|a|=7^9-1$ and $c=1+7+49=57.$ We have that $|G|=3185811$ and $G$ is IBIS by 
Theorem \ref{meta}.

\begin{proof} [Proof of Corollary  \ref{ibisoddpr}]
 Let $V=\soc(G).$ Then $|V|=n$ and a point-stabilizer $G_\omega$ is 
an IBIS irreducible subgroup of $\GL(V)$ of odd order. 
 If $G_\omega$ acts fixed-point-freely on $V$, then $G$ is a Frobenius group. Otherwise, by \cite[Theorem 2.12]{as}, $G_\omega \leq \GaL(1,n)$
and consequently $G\leq \agaL(1,n).$
\end{proof}

\thebibliography{99}
	
\bibitem{b1} T. C. Burness, R. M. Guralnick, J. Sax, On base sizes for symmetric groups, Bull. London Math. Soc. 43 (2011), 386--391.
\bibitem{b2} T. C. Burness, On base sizes for actions of finite classical groups, J. London Math. Soc. 75 (2007), 545--562.
\bibitem{b3} T. C. Burness, M. W. Liebeck, A. Shalev, Base sizes for simple groups and a conjecture of Cameron, Proc. London Math. Soc. 98 (2009), 116--162.
\bibitem{b4} T. C. Burness, E. A. O'Brien, R. A. Wilson, Base sizes for sporadic simple groups, Israel J. Math. 177 (2010), 307--334.	
	
\bibitem{cam} P. J. Cameron, Permutation Groups, London Mathematical Society Student Texts. Cambridge University Press, {1999}.	
	
\bibitem{ibis} P. J. Cameron and D. G. Fon-Der-Flaas, {Bases for Permutation Groups and Matroids}, {Europ. J. Combin.} {16} (1995), 537--544.  
	
	



\bibitem{HP}  Z. Halasi, K. Podoski,  Every coprime linear group admits a base of size two,
  Trans. Amer. Math. Soc. 368 (2016), no. 8, 5857--5887. 

\bibitem{HY} D. Holt and  Y. Yang, Regular orbits of finite primitive solvable groups, the final classification,  J. Symbolic Comput. 116 (2023), 139--145. 

\bibitem{gm} R. Guralnick, K. Magaard,  On the minimal degree of a primitive permutation group, J. Algebra 207 (1998), no. 1, 127--145.	


	
\bibitem{regularorbitnil} B. B. Hargraves, The Existence of Regular Orbits for Nilpotent Groups, J. Algebra  {72} (1981), 54--100. 

\bibitem{lp} M. Lee and P. Spiga, A classification of finite primitive IBIS groups with alternating socle, arXiv:2206.01456

\bibitem{lmm} A. Lucchini, M. Morigi and M. Moscatiello,  Primitive permutation IBIS groups, J. Combin. Theory Ser. A 184 (2021), Paper No. 105516, 17 pp.

\bibitem{yvv} Y. Yang, A Vasil'ev and  E. Vdovin, Regular orbits of finite primitive solvable groups, III. J. Algebra 590 (2022), 139--154.

\bibitem{as} A. Seress, The minimal base size of primitive solvable permutation groups, J. London Math. Soc. (2) 53 (1996), no. 2, 243--255.
	\end{document}